\documentclass[12pt]{amsart}
\usepackage{amsmath, amsthm, amscd, amsfonts}
\usepackage{graphicx}

\newfont{\aj}{eufm10 at12pt}
\newfont{\ajk}{eufm10 at10pt}
\theoremstyle{plain}
\newtheorem{theorem}{Theorem}[section]
\newtheorem{lemma}[theorem]{Lemma}

\newtheorem{corollary}[theorem]{Corollary}
\numberwithin{equation}{section}
\theoremstyle{definition}
\newtheorem{definition}[theorem]{Definition}
\newtheorem{example}[theorem]{Example}

\newtheorem{remark}[theorem]{Remark}

\begin{document}

\title{Parameterized IFS with the Asymptotic Average Shadowing Property}
\author [MEHDI FATEHI NIA ] {MEHDI FATEHI NIA }\\
\address{  Department of Mathematics,
Yazd University, P. O. Box 89195-741 Yazd, Iran \\ \tiny{\textrm{ e-mail:fatehiniam@yazd.ac.ir}}}

\maketitle

\begin{abstract}In this paper we generalize the notion of asymptotic average shadowing property to parameterized \emph{IFS} and prove some related theorems on this notion. Specially, this is proved that every uniformly contracting IFS has the asymptotic average shadowing property.
As an important result,  we show that if a continuous surjective IFS $\mathcal{F}$ on a
compact metric space $X$  has the asymptotic average shadowing
property then $\mathcal{F}$ is chain transitive. Moreover, we introduce some examples and investigate  the relationship between the original asymptotic average shadowing property and asymptotic average shadowing property for IFS. For example,  there is an IFS $\mathcal{F}$  such that  has the asymptotic average shadowing property but does not satisfy the  shadowing property.
\end{abstract}

\section{Introduction}
 The shadowing property of a dynamical system  is one of the most important notions in dynamical systems(see \cite{[AH],[KP],[KS]}). The average shadowing property was further studied by several authors,
with particular emphasis on connections with other notions known from topological dynamics, or more narrowly, shadowing
theory (e.g. see \cite{[Nb],[RG1],[PZ]}). The notion of asymptotic average shadowing property was introduced by Rongbao Gu in \cite{[RA]}. Suppose that $(X,d)$ is a compact metric space and  $f:X\longrightarrow X$  a continuous map. A sequence $\{x_{i}\}_{i\geq 0}$ of points in $ X$ is called an asymptotic average-pseudo-orbit of $f$ if
 $$\lim_{n\rightarrow\infty}\frac{1}{n}\sum_{i=0}^{n-1}d(f(x_{i}),x_{i})=0.$$
 We say that $f$ has the asymptotic average shadowing property if every asymptotic average-pseudo-orbit
 $\{x_{i}\}_{i\geq 0}$ is asymptotically shadowed in average by some point $y\in X$, that is,
 $$\lim_{n\rightarrow\infty}\frac{1}{n}\sum_{i=0}^{n-1}d(f^{i}(y),x_{i})=0.$$  We use the notion of chain recurrent points to study chaotic dynamical systems. The set $CR(f)$ consisting of all chain recurrent pints, i.e., such points $x\in X$ that for any $\delta>0$, there exists a periodic $\delta-$pseudo-orbit through $x$, is called the chain recurrent set of the discrete dynamical system $(X,f)$ \cite{[AH]}.\\Let us recall that a \emph{parameterized Iterated Function System(IFS)} \\ $\mathcal{F}=\{X; f_{\lambda}|\lambda\in\Lambda\}$ is any family of continuous mappings $f_{\lambda}:X\rightarrow X,~\lambda\in \Lambda$, where $\Lambda$ is a finite nonempty set (see\cite{[GG]}).\\ Let $T=\mathbb{Z}$ or $T=\mathbb{Z}_{+}= \{n\in \mathbb{Z}:n\geq0\}$ and $\Lambda^{T}$ denote the set
of all infinite sequences $\{\lambda_{i}\}_{i\in T}$ of symbols belonging to $\Lambda$. A typical element of $\Lambda^{\mathbb{Z}_{+}}$
 can be denoted as $\sigma= \{\lambda_{0},\lambda_{1},...\}$ and we use the shorted notation $$\mathcal{F}_{\sigma_{n}}=f_{\lambda_{n-1}}o ...o f_{\lambda_{0}}.$$ A sequence $\{x_{n}\}_{n\in T}$ in $X$ is called an orbit of the IFS $\mathcal{F}$ if there exist $\sigma\in \Lambda^{T}$ such that $x_{n+1}=f_{\lambda_{n}}(x_{n})$, for each $\lambda_{n}\in \sigma$.\\Iterated function systems( IFS), are used for the construction of deterministic fractals and have found numerous applications, in
particular to image compression and image processing \cite{[B]}. Important notions in
dynamics like attractors, minimality, transitivity, and shadowing can be extended to
IFS (see \cite{[BV],[BVA],[GG],[GG1]}).  \\ A sequence $\{x_{n}\}_{n\in T}$ in $X$ is called a $\delta-$pseudo orbit of $\mathcal{F}$ if there exist $\sigma\in\Lambda^{T}$ such that for every $\lambda_{n}\in \sigma$, we have $d(f_{\lambda_{n}}(x_{n}),x_{n+1})<\delta$.\\One says that the IFS $ \mathcal{F}$ has the \emph{shadowing property }(on $T$) if, given $\epsilon>0$, there exists $\delta>0$ such that for any $\delta-$pseudo orbit $\{x_{n}\}_{n\in T}$ there exist an orbit $\{y_{n}\}_{n\in T}$, satisfying the inequality $d(x_{n},y_{n})\leq \epsilon$ for all $n\in T$. In this case one says that the $\{y_{n}\}_{n\in T}$ or  $y_{0}$, $\epsilon-$ shadows the $\delta-$pseudo orbit $\{x_{n}\}_{n\in T}$.\\The parameterized IFS
$ \mathcal{F}=\{X; f_{\lambda}|\lambda\in\Lambda\}$ is \emph{uniformly contracting} if there exists
$$\beta= sup_{\lambda\in\Lambda} sup_{x\neq y}\frac{d(f_{\lambda}(x),f_{\lambda}(y))}{d(x,y)} $$and this number called also the \emph{contracting ratio},
 is less than one\cite{[GG]}. The authors defined the shadowing property for a parameterized iterated function system and prove that if a parameterized IFS is uniformly contracting or expanding, then it has the shadowing property\cite{[GG]}. Throughout this paper, we assume that $X$ is a compact metric space, $\Lambda$ is a finite nonempty set and $T=\mathbb{Z}_{+}$.\\
 The present paper concerns the asymptotic average shadowing property for parameterized IFS and some important result about asymptotic average shadowing property are extended to iterated function systems. In this direction some proofs and example are based  on the previous work of Gu (\cite{[RA]}). Specially; Theorem \ref{tt}, Theorem \ref{tr} and Example \ref{e9} and we thank him for that. In Theorem \ref{tc} we prove that each uniformly contracting parameterized IFS has the asymptotic average shadowing property on $\mathbb{Z}_{+}$. Theorem \ref{td} shows that the asymptotic average shadowing property is preserved under topological conjugacy and Theorem \ref{tt} prove that a parameterized IFS  $\mathcal{F}$ has this property (on $\mathbb{Z}_{+}$) if and only if  so does $ \mathcal{F}^{k}$. Where $k\geq 2$ and  $\mathcal{F}^{k}$ is the $k$-fold composition of  $\mathcal{F}$. In Section 3, we prove that if $ \mathcal{F}$ has the asymptotic average shadowing property then $ \mathcal{F}$ is chain transitive,  which is one of the main results of this
paper. Finally, as an example,  we introduce a nontrivial  parameterized IFS which has the asymptotic shadowing property.
 \section{\textbf{ Asymptotic Average Shadowing Property for Iterated Function Systems}}
In this section we investigate the structure of parameterized \emph{IFS} with the asymptotic average shadowing property.
\begin{definition} A sequence $\{x_{i}\}_{i\geq 0}$ of points in $X$ is called an asymptotic average-pseudo-orbit of $ \mathcal{F}$ if there exists a natural number $N = N(\delta) > 0$ and  $\sigma=\{\lambda_{0},\lambda_{1},\lambda_{2},...\}$ in  $\Lambda^{\mathbb{Z}_{+}}$,
such that
$$\lim_{n\rightarrow\infty}\frac{1}{n}\sum_{i=0}^{n-1}d(f_{\lambda_{i}}(x_{i}),x_{i+1})=0.$$
A sequence $\{x_{i}\}_{i\geq 0}$ in $X$ is said to be asymptotically shadowed in average by a point $z$ in $X$ if there exist $\sigma\in\Lambda^{\mathbb{Z}_{+}}$ such that
$$\lim_{n\rightarrow\infty}\frac{1}{n}\sum_{i=0}^{n-1}d(\mathcal{F}_{\sigma_{i}}(z),x_{i})=0.$$ \end{definition}
A continuous IFS $\mathcal{F}$ is said to have the  asymptotic average shadowing property if  every asymptotic average-pseudo-orbit of $\mathcal{F}$ is  asymptotically shadowed in average by some point in $X$.
Please note that if $\Lambda$ is a set with one member then the parameterized IFS $ \mathcal{F}$ is an ordinary discrete dynamical system.\\In this case the asymptotic average shadowing property for $\mathcal{F}$ is ordinary asymptotic average shadowing property for a discrete dynamical system.
\begin{example}
Let $X=\{a_{1},...,a_{n}\}$ be a finite set with the discrete metric $d$. Suppose $\{f_{\lambda}\}_{\lambda\in\Lambda}$ is the family of all surjective functions on $X$. For each arbitrary sequence $\{x_{i}\}_{i\geq 0}$ there exists $z\in X$ and   $\sigma\in \Lambda^{\mathbb{Z}_{+}}$ such that $\mathcal{F}_{\sigma_{i}}(z)= x_{i}$ for all $i\geq 0$. Then  $ \mathcal{F}=\{X; f_{\lambda}|\lambda\in\Lambda\}$ has the asymptotic average shadowing property.
\end{example}
In \cite{[GG]}, the authors  have proved that every uniformly contracting \emph{IFS} has the shadowing property. In the following theorem we prove that every uniformly contracting \emph{IFS} has the asymptotic average shadowing property.
\begin{theorem}\label{tc}
If a parameterized IFS $\mathcal{F}=\{X; f_{\lambda}|\lambda\in\Lambda\}$ is uniformly contracting, then it has the asymptotic average shadowing property.
\end{theorem}
\begin{proof}
Assume that $\beta<1$ is the contracting ratio of $\mathcal{F}$ and suppose $\{x_{i}\}_{i\geq 0}$ is an asymptotic average pseudo orbit for $ \mathcal{F}$. So there exist\\ $\sigma=\{\lambda_{0},\lambda_{1},\lambda_{2},...\}\in\Lambda^{\mathbb{Z}_{+}}$ such that $\lim_{n\rightarrow\infty}\frac{1}{n}\sum_{i=0}^{n-1}d(f_{\lambda_{i}}(x_{i}),x_{i+1})=0$. Put $\alpha_{i}=d(f_{\lambda_{i}}(x_{i}),x_{i+1})$, for all $i\geq 0$. Consider an orbit $\{y_{i}\}_{i\geq 0}$ such that $y_{0}\in X$ and $y_{i+1}=f_{\lambda_{i}}(y_{i})$, for all $i\geq 0$.\\Now we show that $\lim_{n\rightarrow\infty}\frac{1}{n}\sum_{i=0}^{n-1}d(y_{i},x_{i})=0$.\\
Take $M=d(x_{0},y_{0})$. Obviously, $$d(x_{1},y_{1})\leq d(x_{1},f_{\lambda_{0}}(x_{0}))+d(f_{\lambda_{0}}(x_{0}),f_{\lambda_{0}}(y_{0}))\leq \alpha_{0}+\beta M.$$
Similarly \begin{eqnarray*}d(x_{2},y_{2})&\leq& d(x_{2},f_{\lambda_{1}}(x_{1}))+d(f_{\lambda_{1}}(x_{1}),f_{\lambda_{1}}(y_{1}))\\&\leq& \alpha_{1}+\beta d(x_{1},y_{1})\\&\leq& \alpha_{1}+\beta(\alpha_{0}+\beta M).\end{eqnarray*} And
\begin{eqnarray*}d(x_{3},y_{3})&\leq& d(x_{3},f_{\lambda_{2}}(x_{2}))+d(f_{\lambda_{2}}(x_{2}),f_{\lambda_{2}}(y_{2}))\\ &\leq& \alpha_{2}+\beta d(x_{2},y_{2})\\ &\leq& \alpha_{2}+\beta(\alpha_{1}+\beta d(x_{1},y_{1}))\\ &\leq& \alpha_{2}+\beta(\alpha_{1}+\beta (\alpha_{0}+\beta M)\\&=&\alpha_{2}+\beta\alpha_{1}+\beta^{2} \alpha_{0}+\beta^{3} M.\end{eqnarray*}
By induction, one can prove that for each $i>2$
$$d(x_{i},y_{i})\leq \alpha_{i-1}+\beta\alpha_{i-2}+...+\beta^{i-1} \alpha_{0}+\beta^{i} M.$$
This implies that \begin{eqnarray*}\sum_{i=0}^{n-1}d(y_{i},x_{i})
&\leq &M(1+\beta+..+\beta^{n-1})\\& +&\alpha_{0}(1+\beta+..+\beta^{n-2})\\ &+&\alpha_{1}(1+\beta+..+\beta^{n-3})\\& \vdots&\\&+&\alpha_{n-2}\\&\leq&
\frac{1}{1-\beta}(M+\sum_{i=0}^{n-2}\alpha_{i}).\end{eqnarray*}
Therefore \begin{eqnarray*}\lim_{n\rightarrow\infty}\frac{1}{n}\sum_{i=0}^{n-1}d(y_{i},x_{i})&\leq&\lim_{n\rightarrow\infty}\frac{1}{n}( \frac{1}{1-\beta}(M+\sum_{i=0}^{n-2}\alpha_{i}))=0.\end{eqnarray*}
So the proof is complete.
\end{proof}
Let us recall some notions related to symbolic dynamics. \\Let $\Sigma_{2}\{(s_{0}s_{1}s_{2}...)|s_{i}=0 or ~1\}$. We will
refer to elements of $\Sigma_{2}$ as points in $\Sigma_{2}$. Let $s=s_{0}s_{1}s_{2}...$ and $t=t_{0}t_{1}t_{2}...$ be points  in $\Sigma_{2}$. We denote the distance between $s$ and $t$ as $d(s,t)$ and define it by \\$d(s,t) =\left\lbrace
 \begin{array}{c l}
  0,& \text{ $s=t$}\\
  \frac{1}{2^{k-1}},& \text{$k=min\{i; s_{i}\neq t_{i}\}$}\end{array}
\right.$
\begin{example}\label{e2}
Let $f_{0},~f_{1}:\Sigma_{2}\rightarrow\Sigma_{2}$ are two map defined as \\$f_{0}(s_{0}s_{1}s_{2}...)=0s_{0}s_{1}s_{2}...$ and $f_{1}(s_{0}s_{1}s_{2}...)=1s_{0}s_{1}s_{2}...$ for each $s=s_{0}s_{1}s_{2}...\in\Sigma_{2}$.
\\ This is clear that $\mathcal{F}=\{\Sigma_{2}; f_{0}, f_{1}\} $ is an uniformly contracting IFS and by Theorem \ref{tc} has the asymptotic average shadowing property.
\end{example}
  \begin{definition}
  Suppose $(X,d)$ and $(Y,d^{'})$ are compact metric spaces and $\Lambda$ is a finite set. Let $ \mathcal{F}=\{X; f_{\lambda}|\lambda\in\Lambda\}$ and $ \mathcal{G}=\{Y; g_{\lambda}|\lambda\in\Lambda\}$ are two $IFS$ which $f_{\lambda}:X\rightarrow X$ and  $g_{\lambda}:Y\rightarrow Y$ are continuous maps for all $\lambda\in\Lambda$. We say that $ \mathcal{F}$ is topologically conjugate to $ \mathcal{G}$  if there is a homeomorphism  $h:X\rightarrow Y$  such that $g_{\lambda}=h o f_{\lambda} o h^{-1}$, for all $\lambda\in\Lambda$. In this case, $h$ is called a topological conjugacy. \end{definition}
 It is well known that if $f:X\rightarrow X$ and $g:Y\rightarrow Y$ are conjugated then $f$ has the shadowing property if and only if so does $g$ \cite{[AH]}. In the next theorem we show that asymptotic average shadowing property is invariant under topological conjugacy  for \emph{iterated function systems }.
\begin{theorem}\label{td}
Suppose $(X,d)$ and $(Y,d^{'})$ are compact metric spaces and $\Lambda$ is a finite set. Let $ \mathcal{F}=\{X; f_{\lambda}|\lambda\in\Lambda\}$ and $ \mathcal{G}=\{Y; g_{\lambda}|\lambda\in\Lambda\}$ are two conjugated IFS with topological conjugacy $h$. Then $\mathcal{F}$ has the asymptotic average shadowing property if and only if so does $\mathcal{G}$.
\end{theorem}
\begin{proof}
 Suppose  $\{y_{i}\}_{i\geq 0}$ is an asymptotic average pseudo orbit of $\mathcal{G}$. This means that there exist $\sigma=\{\lambda_{0},\lambda_{1},\lambda_{2},...\}$ in  $\Lambda^{\mathbb{Z}_{+}}$,
such that, $$\lim_{n\rightarrow\infty}\frac{1}{n}\sum_{i=0}^{n-1}d^{'}(g_{\lambda_{i}}(y_{i}),y_{i+1})=0.$$ Put $x_{i}=h^{-1}(y_{i})$, for all $i\geq 0$. Then $$d(f_{\lambda_{i}}(x_{i}),x_{i+1})=d(h^{-1} o g_{\lambda_{i}}(y_{i}),h^{-1}(y_{i+1})),$$ for all $i\geq 0$. Thus, since $h$ is a homeomorphism, we have $$\lim_{n\rightarrow\infty}\frac{1}{n}\sum_{i=0}^{n-1}d(f_{\lambda_{i}}(x_{i}),x_{i+1}))=0.$$
 Hence $\{x_{i}\}_{i\geq 0}$ is an asymptotic average pseudo orbit for $\mathcal{F}$ and so there is an orbit $\{z_{i}\}_{i\geq 0}$ of  $\mathcal{F}$ such that $$\lim_{n\rightarrow\infty}\frac{1}{n}\sum_{i=0}^{n-1}d(z_{i},x_{i})=0.$$ Note that for each $i\geq 0$ there is $\mu_{i}\in \Lambda$ such that $z_{i+1}=f_{\mu_{i}}(z_{i})$. Let $w_{i}=h(z_{i})$, this is clear that \\ $w_{i+1}=h(z_{i+1})=h(f_{\mu_{i}}(z_{i}))=g_{\mu_{i}}(h(z_{i}))=g_{\mu_{i}}(w_{i})$ for all $i\geq 0$.\\ Therefor $\{w_{i}\}_{i\geq 0}$ is an orbit of $\mathcal{G}$ and $d^{'}(w_{i},y_{i})=d^{'}(h(z_{i}),h(x_{i})),$ for all $i\geq 0$.  Then $$\lim_{n\rightarrow\infty}\frac{1}{n}\sum_{i=0}^{n-1}d^{'}(w_{i},y_{i})=
 \lim_{n\rightarrow\infty}\frac{1}{n}\sum_{i=0}^{n-1}d^{'}(h(z_{i}),h(x_{i}))=0.$$
\end{proof}
Gu \cite{[RA]} showed that, for a dynamical system $(X,f)$, If $f$ has the average-shadowing property, then so does $f^{k}$ for every $k\in \mathbb{N}$. The following theorem generalize a similar result for IFS.\\
We use the following lemma to prove Theorem \ref{tt}
\begin{lemma}\cite{[W]}\label{l1}
  If $\{a_{i}\}_{i=0}^{\infty}$ is a bounded sequence of non-negative real numbers, then the following are equivalent:\\$i-\lim_{n\rightarrow\infty}\frac{1}{n}\sum_{i=0}^{n-1}a_{i}=0.$\\
  $ii-$ There is a subset $J$ of $\mathbb{Z}^{+}$ of density zero\\
  $(i.e., \lim_{n\rightarrow\infty}\frac{Card(J\cap \{0,1,...,N-=1\})}{n}=0)$,\\such that $\lim_{j\rightarrow\infty} a_{j}=0$ provided $j\notin J$, where $\mathbb{Z}^{+}$ is the set of all non-negative integers.
  \end{lemma}

\begin{theorem}\label{tt}
Let $\Lambda$ be a finite set, $ \mathcal{F}=\{X; f_{\lambda}|\lambda\in\Lambda\}$ is an $IFS$ and let $k\geq 2$ be an integer. Set $$\mathcal{F}^{k}= \{X; g_{\mu}|\mu\in\Pi\}=\{X; f_{\lambda_{k-1}}o ...of_{\lambda_{0}}|\lambda_{0},...,\lambda_{k-1}\in\Lambda\}.$$
 $\mathcal{F}$ has the asymptotic average shadowing property  if and only if  so does $ \mathcal{F}^{k}$.
\end{theorem}
\begin{proof} Suppose  that $\mathcal{F}$ has the asymptotic average shadowing property and $\{x_{i}\}_{i\geq 0}$ is an asymptotic average pseudo orbit for $\mathcal{F}^{k}$. So there exist\\ $\gamma=\{\mu_{0}, \mu_{1},\mu_{2},\mu_{3},...\}$ in  $\Pi^{\mathbb{Z}_{+}}$,
such that
$$\lim_{n\rightarrow\infty}\frac{1}{n}\sum_{i=0}^{n-1}d(g_{\mu_{i}}(x_{i}),x_{i+1})=0.$$ Now we put $y_{nk+j}=f_{\lambda^{n}_{j-1}}o ...of_{\lambda^{n}_{0}}(x_{n})$ and
$y_{nk}=x_{n}$,
for $0<j<k$ and $n\in\mathbb{Z}_{+}$. Which  $g_{\mu_{n}}=f_{\lambda^{n}_{k-1}}o ...of_{\lambda^{n}_{0}}$,
 that is $$\{y_{i}\}_{i\geq0}=\{x_{0},f_{\lambda^{0}_{0}}(x_{0}),...,f_{\lambda^{0}_{k-2}}o ...of_{\lambda^{0}_{0}}(x_{0}),x_{1},f_{\lambda^{1}_{0}}(x_{1}),...,f_{\lambda^{1}_{k-2}}o ...of_{\lambda^{1}_{0}}(x_{1}),....\}.$$
 This is clear that $\{y_{i}\}_{i\geq0}$ is an asymptotic average pseudo-orbit for $\mathcal{F}$. So there exists  $x\in X$ and $\sigma\in \Lambda^{\mathbb{Z}_{+}}$  satisfies, $$\lim_{n\rightarrow\infty}\frac{1}{n}\sum_{i=0}^{n-1}d(\mathcal{F}_{\sigma_{i}}(x),y_{i})=0.$$ This implies that $$
 \lim_{n\rightarrow\infty}\frac{1}{kn}\sum_{i=0}^{n-1}d(\mathcal{F}_{\sigma_{ki}}(x),y_{ki})\leq
 \lim_{n\rightarrow\infty}\frac{1}{kn}\sum_{j=0}^{kn-1}d(\mathcal{F}_{\sigma_{j}}(x),y_{j})=0.$$
 So $$\lim_{n\rightarrow\infty}\frac{1}{n}\sum_{i=0}^{n-1}d(\mathcal{F}_{\sigma_{ki}}(x),y_{ki})=0.$$
 Since $y_{ik}=x_{i}$, this statement is proved.\\
 Conversely, suppose $\mathcal{F}^{k}$ has the asymptotic average shadowing property for some integer $k>0$. For each $\epsilon>0$ we can find a $\delta\in(0, \frac{\epsilon}{k})$ such that $d(s,t)<\delta$ implies
 \begin{equation}
\label{cb}d(f_{\lambda_{l}}o...of_{\lambda_{0}}(s),f_{\lambda_{l}}o...of_{\lambda_{0}}(t))<\frac{\epsilon}{k},\end{equation}
 for all $0\leq l\leq k-1$ and $\lambda_{l},...,\lambda_{0}\in\Lambda$.\\ Consider  $\{x_{i}\}_{i\geq 0}$ as an asymptotic average pseudo orbit of $ \mathcal{F}$.
 So there exist a sequence $\{\lambda_{0},\lambda_{1},...\}$ in $\Lambda^{\mathbb{Z}_{+}}$
 such that $$\lim_{n\rightarrow\infty}\frac{1}{kn}\sum_{i=0}^{n-1}d(f_{\lambda_{i}}(x_{i}),x_{i+1})=0.$$
  Lemma \ref{l1} implies that there is a set $I_{0}\subset \mathbb{Z}_{+}$ of zero density such that \\$\lim_{j\rightarrow \infty} d(f_{\lambda_{j}}(x_{j}),x_{j+1})=0$ provided $j \notin I_{0}$.
 \\ Let $I_{1}=\{j:\{jk,jk+1,...,jk+k-1\}\cap I_{0}\neq \emptyset\}$ and\\ $I=\bigcup_{j\in I_{1}}\{jk,jk+1,...,jk+k-1\}.$ Then $I$ has density zero and \\ $\lim_{j\rightarrow \infty} d(f_{\lambda_{j}}(x_{j}),x_{j+1})=0$ provided $j \notin I$. For the above $\delta>0$, there is an integer $N_{1}>0$ such that $d(f_{\lambda_{j}}(x_{j}),x_{j+1})<\delta$ for all $j>N_{1}$ and $j\notin I$. It follows that  $d(f_{\lambda_{jk+l}}(x_{jk+l}),x_{jk+l+1})<\delta$ for all $j>N_{1}$, $0\leq l< k$ and $j\notin I$.
  Therefor, by (\ref{cb})
   we have $d(f_{\lambda_{k-1}}o...of_{\lambda_{0}}(x_{jk}),x_{(j+1)k})<\epsilon$, for all $j>N_{1}$, $0\leq l< k$ and $j\notin I$. Then, $\lim_{j\rightarrow \infty}d(f_{\lambda_{k-1}}o...of_{\lambda_{0}}(x_{jk}),x_{(j+1)k})=0$ provided $j\notin I$. By Lemma \ref{l1}, $\{x_{ik}\}_{i\geq 0}$ is an asymptotic average pseudo orbit of $\mathcal{F}^{k}$. Since $\mathcal{F}^{k}$ has the asymptotic average shadowing property, there exists the sequence  $\gamma=\{\omega_{1},\omega_{2},...\}\in \Pi^{\mathbb{Z}_{+}}$ and point $u\in X$
  such that $\lim_{n\rightarrow\infty}\frac{1}{n}\sum_{i=0}^{n-1}d(\mathcal{F}^{k}_{\gamma_{i}}(u),x_{ik})=0$.
  By Lemma \ref{l1} again, there is a set $J_{0}\subset\mathbb{Z}_{+}$ of zero density such that  $\lim_{j\rightarrow\infty}d(\mathcal{F}^{k}_{\gamma_{j}}(u),x_{jk})=0$ provided $j\notin J_{0}$. Let $J_{1}=\{j:\{jk,jk+1,...,jk+k-1\}\cap J_{0}\neq \emptyset\}$
 and $J=\bigcup_{j\in J_{1}}\{jk,jk+1,...,jk+k-1\}.$ Then both $J_{1}$ and $J$ have density zero and  $\lim_{j\rightarrow\infty}d(\mathcal{F}^{k}_{\gamma_{j}}(u),x_{jk})=0$ provided $j\notin J_{1}$.\\
 For the above $\delta>0$, there is an integer $N_{2}>0$ such that $d(\mathcal{F}^{k}_{\gamma_{j}}(u),x_{jk})<\delta$ for all $j>N_{2}$ and $j\notin J_{1}.$ Then
 $$d(\mathcal{F}_{\sigma_{jk+l}}(u),x_{jk+l})<\epsilon,~~~~0\leq l<k$$
 for all $j>N_{2}$ and $j\notin J_{1}.$ Where $\sigma=\{\nu_{1},\nu_{2},...\}$ that $g_{\omega_{j}}=f_{\nu_{jk-1}}o ...of_{\nu_{jk}}$, for all $j>0$. Thus we have $\lim_{j\rightarrow \infty}d(\mathcal{F}_{\sigma_{j}}(u),x_{j})=0$ provided $j\notin J$. By using of Lemma \ref{l1} we can deduce that $\{x_{i}\}_{i>0}$ is asymptotically shadowed in average by the point $z$. This prove that $\mathcal{F}$ has the asymptotic average shadowing property.
  \end{proof}
\begin{example}
Consider the IFS  $\mathcal{F}$ in Example \ref{e2}, by Theorem \ref{tt} IFS $\mathcal{F}^{k}$ also have the asymptotic average shadowing property, for all $k>1$.\\
For example if $k=2$ then $\mathcal{F}^{2}=\{\Sigma_{2}; g_{0}, g_{1},g_{2},g_{3}\} $ where
$$g_{0}(s_{0}s_{1}s_{2}...)=f_{0}of_{0}(s_{0}s_{1}s_{2}...)=00s_{0}s_{1}s_{2}....,$$
$$g_{1}(s_{0}s_{1}s_{2}...)=f_{1}of_{0}(s_{0}s_{1}s_{2}...)=10s_{0}s_{1}s_{2}....,$$
$$g_{2}(s_{0}s_{1}s_{2}...)=f_{0}of_{1}(s_{0}s_{1}s_{2}...)=01s_{0}s_{1}s_{2}....,$$
$$g_{3}(s_{0}s_{1}s_{2}...)=f_{1}of_{1}(s_{0}s_{1}s_{2}...)=11s_{0}s_{1}s_{2}...$$
for each $s=s_{0}s_{1}s_{2}...\in\Sigma_{2}$.\\
Please note that $\mathcal{F}^{k}$ also is a uniformly contracting IFS.
\end{example}
Let $(X,d)$ and $(Y,d^{'}) $ are two complete metric spaces. Consider the product set $X\times Y$ endowed with the metric $D((x_{1},y_{1}),(x_{2},y_{2}))=max\{d(x_{1},x_{2}), d^{'}(y_{1},y_{2})\}$.\\
 Let $ \mathcal{F}=\{X; f_{\lambda}|\lambda\in\Lambda\}$ and $ \mathcal{G}=\{Y; g_{\gamma}|\gamma\in\Gamma\}$ are two parameterized IFS. The IFS $\mathcal{H}=\mathcal{F}\times\mathcal{G}=\{X\times Y; f_{\lambda}\times g_{\gamma}:\lambda\in\Lambda,\gamma\in\Gamma\}$, defined by
 $( f_{\lambda}\times g_{\gamma})(x,y)=( f_{\lambda}(x), g_{\gamma}(y))$ is called the \emph{product} of the IFS $\mathcal{F}$ and $\mathcal{G}$.
  \begin{theorem}\label{tp}
The product of two parameterized IFS has the asymptotic average shadowing property if and only if each projection has.
  \end{theorem}
  \begin{proof}
  Suppose $\{(x_{i},y_{i})\}_{i\geq0}$ is a an asymptotic average pseudo orbit for $\mathcal{F}\times\mathcal{G}$. By definition of the asymptotic average pseudo orbit, $\{x_{i}\}_{i\geq0}$ and $\{y_{i}\}_{i\geq0}$ are asymptotic average pseudo orbits for $\mathcal{F}$ and $\mathcal{G}$ respectively. Then there exists the sequences $\sigma\in \Lambda^{\mathbb{Z}_{+}}$, $\gamma\in \Gamma^{\mathbb{Z}_{+}}$ and points $u\in X$, $v\in Y$
  such that $$\lim_{n\rightarrow\infty}\frac{1}{n}\sum_{i=0}^{n-1}d(\mathcal{F}_{\sigma_{i}}(u),x_{i})=0$$ and $$\lim_{n\rightarrow\infty}\frac{1}{n}\sum_{i=0}^{n-1}d^{'}(\mathcal{G}_{\gamma_{i}}(v),y_{i})=0.$$ Now, Lemma 3.4 of \cite{[Nb]} implies that  $$\lim_{n\rightarrow\infty}\frac{1}{n}\sum_{i=0}^{n-1}D(\mathcal{F}_{\sigma_{i}}\times\mathcal{G}_{\gamma_{i}}(u,v),(x_{i},y_{i}))\\= \lim_{n\rightarrow\infty}\frac{1}{n}\sum_{i=0}^{n-1}max(d(\mathcal{F}_{\sigma_{i}}(u),x_{i}),d^{'}(\mathcal{G}_{\gamma_{i}}(v),y_{i}))=0.$$
  So $\mathcal{F}\times\mathcal{G}$ has the asymptotic average shadowing property.\\
  Conversely; suppose that $\mathcal{F}\times\mathcal{G}$ has the asymptotic average shadowing property. Let $\{x_{i}\}_{i\geq0}$ be an asymptotic average pseudo orbit for $\mathcal{F}$. Take an orbit $\{y_{i}\}_{i\geq0}$ in $\mathcal{G}$. So $\{(x_{i},y_{i})\}_{i\geq0}$ is an asymptotic average pseudo orbit for $\mathcal{F}\times\mathcal{G}$. Then there exists the sequences $\sigma\in \Lambda^{\mathbb{Z}_{+}}$, $\gamma\in \Gamma^{\mathbb{Z}_{+}}$ and point $(u,v)\in X\times Y$ such that  $\lim_{n\rightarrow\infty}\frac{1}{n}\sum_{i=0}^{n-1} D(\mathcal{F}_{\sigma_{i}}\times\mathcal{G}_{\gamma_{i}}(u,v),(x_{i},y_{i}))=0$. \\Therefore
  $\lim_{n\rightarrow\infty}\frac{1}{n}\sum_{i=0}^{n-1}d(\mathcal{F}_{\sigma_{i}}(u), x_{i})=0$.\end{proof}

\section{The asymptotic average shadowing property and chain transitivity}
If $b<\infty$, then we say that a finite $\delta-$pseudo-orbit $\{x_{i}\}_{i=0}^{b}$ of $ \mathcal{F}$ is a $\delta-$chain from $x_{0}$ to $x_{b}$ with length $b+1$. \\A non-empty subset $L$ of $X$ is said to be chain transitive if for any $x, y\in L$ and any $\delta>0$ there is a $\delta-$chain of $ \mathcal{F}$ from $x$ to $y$.
We recall that, a point $x\in X$ is a \emph{chain recurrent} for $\mathcal{F}$ if for every $\epsilon>0$ there exists a finite sequences of points $\{p_{i}\in X :i=0,1,...,n\}$ with $p_{0}=p_{n}=x$, and a corresponding sequence of indices $\{\lambda_{i}\in\Lambda:i=1,2,...,n\}$ satisfying $d(f_{\lambda_{i}}(p_{i}),p_{i+1})\leq \epsilon$ for $i=1,2,...,n-1$. Such a sequence of points is called an $\epsilon-$chain from $x$ to $x$, similarly we can define an $\epsilon-$chain from $x$ to $y$ \cite{[BWL]}.
\\
An IFS $ \mathcal{F}$ is said to be chain transitive if $X$ is a chain transitive set.\\
  The following theorem is one of the main results of the paper.
\begin{theorem}\label{tr}
Let $X$ be a compact metric space and $ \mathcal{F}=\{X; f_{\lambda}|\lambda\in\Lambda\}$ be an IFS such that each $f_{\lambda}$ is surjective. If $ \mathcal{F}$ has the asymptotic average shadowing property then $ \mathcal{F}$ is chain transitive.
\end{theorem}
\begin{proof}
Suppose $\delta$ is a positive number and $x,y\in X$ are two distinct points. It is sufficient to find a $\delta-$chain from $x$ to $y$. \\
Let $g$ be an arbitrary function of $ \mathcal{F}$.  Define a sequence $\{t_{i}\}_{i\geq 0}$ in $X$ as follows. \\$t_{0}=x,~~~y=t_{1}$\\
$t_{2}=x,~~~y=t_{3}$\\
$t_{4}=x,~~~,g(x),~y_{-1},~y=t_{7}$\\
$~~~~~~~~~~~\vdots$\\
$t_{2^{k}}=x,~~g(x),...,~~~g^{2^{k-1}-1}(x),~y_{-2^{k-1}+1},~...,~y_{-1},~~~y=t_{2^{k+1}-1}$\\
$~~~~~~~~~~~\vdots$\\
where $g(y_{-j})=y_{-j+1}$ for every $j>0$ and $y_{0}=y.$ This is clear that for $2^{k}\leq n
<2^{k+1}$,$\frac{1}{n}\sum_{i=0}^{n-1}d(g(t_{i}),t_{i+1})<\frac{2(k+1)\times D}{2^{k}},$where\\ $D=\max\{d(x,y):x,y\in X\}$. So $\lim_{n\rightarrow\infty}\frac{1}{n}\sum_{i=0}^{n-1}d(g(t_{i}),t_{i+1})=0$. Then $\{t_{i}\}_{i\geq 0}$ is an asymptotic average pseudo orbit of $ \mathcal{F}$. Since $ \mathcal{F}$ has the asymptotic average shadowing property, there is a point $w$ in $X$ and a sequence $\sigma\in \Lambda^{\mathbb{Z}_{+}}$, such that \begin{equation}
\label{ffa}
\lim_{n\rightarrow\infty}\frac{1}{n}\sum_{i=0}^{n-1}d^{'}(\mathcal{F}_{\sigma_{i}}(w),t_{i})=0.\end{equation}
Consider  the above positive number $\delta$, there is an $\eta\in(0,\delta)$ such that\\ $d(u,v)<\eta$ implies $d(f_{\lambda}(u),f_{\lambda}(v))<\delta $, for all $u,v\in X$ and $\lambda\in\Lambda$.\\
Then we have the following result.\\
\textbf{Claim.}$(1)$ \emph{There exist infinitely many positive integers $j$ such that \\$t_{n_{j}}\in\{x,g(x),...,g^{2^{j}-1}(x)\}$
and $d(\mathcal{F}_{\sigma_{n_{j}}}(w),t_{n_{j}})<\eta$;\\
$(2)$ There exist infinitely many positive integers $l$ such that \\
$t_{n_{l}}\in\{y_{-2^{l}+1},...,y_{-1},y\}$
and $d(\mathcal{F}_{\sigma_{n_{l}}}(w),t_{i})<\eta$.\\}
So, we can choose two positive integers $j_{0}$ and $l_{0}$ such that $n_{j_{0}}<n_{l_{0}} $ and \\$t_{n_{j_{0}}}\in\{x,g(x),...,g^{2^{j_{0}}-1}(x)\}$ and $d(\mathcal{F}_{\sigma_{n_{j_{0}}}}(w),t_{n_{j_{0}}})<\eta$;\\$t_{n_{l_{0}}}\in\{y_{-2^{l}+1},...,y_{-1},y\}$
and $d(\mathcal{F}_{\sigma_{n_{l_{0}}}}(w),t_{n_{l_{0}}})<\eta$.\\
Let
\\
$t_{n_{j_{0}}}=g^{j_{1}}(x)$ for some $j_{1}>0$;\\
$t_{n_{l_{0}}}=y_{-l_{1}}$ for some $l_{1}>0$.\\Since $d(\mathcal{F}_{\sigma_{n_{j_{0}}}}(w),t_{n_{j_{0}}})<\eta$ then $d(\mathcal{F}_{\sigma_{n_{j_{0}}+1}}(w),f_{\lambda_{n_{j_{0}}+1}}(t_{n_{j_{0}}}))<\delta$. Similarly, since  $d(\mathcal{F}_{\sigma_{n_{l_{0}}}}(w),t_{n_{l_{0}}})<\eta<\delta$ then we have a $\delta-$chain from $x$ to $y$ as follows:\\
$x,g(x),...,g^{j_{1}}(x)=t_{n_{j_{0}}},$\\ $\mathcal{F}_{\sigma_{n_{j_{0}}+1}}(w),\mathcal{F}_{\sigma_{n_{j_{0}}+2}}(w),...,\mathcal{F}_{\sigma_{n_{l_{0}}-1}}(w),$\\$
t_{n_{l_{0}}}=y_{-l_{1}},y_{-l_{1}+1},...,y.$\\
Therefore $\mathcal{F}$ is chain transitive.
\end{proof}
\textbf{Proof of Claim:}With no loss of generality we prove only the conclusion $(1)$.\\
By contrary, suppose that there is a positive integer $N$ such that for all integers $k>N$, we have \\$t_{i}\in\{x,g(x),...,g^{2^{j}-1}(x)\}$,\\
then $d(\mathcal{F}_{\sigma_{i}}(w),t_{i}))\geq\eta$.\\
This implies that\\$\lim \inf_{n\rightarrow\infty}\frac{1}{n}\sum_{i=0}^{n-1}d^{'}(\mathcal{F}_{\sigma_{i}}(w),t_{i})\geq\frac{\eta}{2}.$ That is contradicts with (\ref{ffa}).
\\The next
corollary quickly follows from Theorem \ref{tr}.
\begin{corollary}
Let $X$ be a compact metric space and $ \mathcal{F}=\{X; f_{\lambda}|\lambda\in\Lambda\}$ be an IFS such that each $f_{\lambda}$ is surjective. If $ \mathcal{F}$ has the asymptotic average shadowing property then every point $x\in X$ is a chain recurrent point, that is, $CR(\mathcal{F})=X.$
\end{corollary}
\begin{remark}
Let $ \mathcal{F}=\{X; f_{\lambda}|\lambda\in\Lambda\}$ and $\lambda_{0}\in\Lambda$. Every chain recurrent of $ f_{\lambda_{0}}$ is a chain recurrent of $\mathcal{F}$ but by Example \ref{e8} conversely is not true.
\end{remark}
Consider a circle $S^{1}$ with coordinate $x \in [0; 1)$ and we denote by $d$
the metric on $S^{1}$ induced by the usual distance on the real line. Let $\pi(x) :\mathbb{ R} \rightarrow S^{1}$ be the covering projection defined by the relations
$$\pi(x) \in [0; 1) and~ \pi(x) = x(x mod 1)$$
with respect to the considered coordinates on $S^{1}$.
\begin{example}\label{e8}
Let $\pi:[0,1]\rightarrow S^{1}$ be a map defined by $\pi(t)=(\cos(2\pi t),\sin(2\pi t))$
 Let $F_{1}:[0,1]\rightarrow[0,1]$ be a homeomorphism defined by \\
 $F_{1}(t) =\left\lbrace
 \begin{array}{c l}
  t+(\frac{1}{2}-t)t& \text{if $0\leq t\leq \frac{1}{2}$}\\
  t-(t-\frac{1}{2})(1-t)& \text{if $\frac{1}{2}\leq t\leq 1$} \end{array}
\right. $ \\
And
 $F_{2}:[0,1]\rightarrow[0,1]$ be a homeomorphism defined by \\
 $F_{2}(t) =\left\lbrace
 \begin{array}{c l}
 t+(\frac{1}{2}-t)t& \text{if $0\leq t\leq \frac{1}{2}$}\\
  t+(1-t)(t-\frac{1}{2})& \text{if $\frac{1}{2}\leq t\leq 1$} \end{array}
\right. $\\
Assume that $f_{i}$ is homeomorphisms on $S^{1}$ defined by\\ $f_{i}(\cos(2\pi t),\sin(2\pi t))=(\cos(2\pi F_{i}(t)),\sin(2\pi F_{i}(t)))$, for $i\in\{0,1\}$. \\Consider  $\mathcal{F}=\{S^{1}, f_{\lambda}|\lambda\in\{0,1\}\}$, this is clear that  $x=\pi(\frac{1}{2})$ is not a chain recurrent point for $f_{1}$.
 We show that  $x$  is a chain recurrent for $\mathcal{F}.$  Given $\epsilon>0$, there exist $\delta>0$ such that $|s-t|<\delta$ implies $d(\pi(s),\pi(t))<\epsilon.$ By definition of $F_{1}$ this is clear that if $0<t<a$ then $\{F_{1}^{n}(t)\}_{n\geq 0}$ is an increasing sequence that converges to $\frac{1}{2}$. Similarly if $\frac{1}{2}<t<1$ then $\{F_{2}^{n}(t)\}_{n\geq 0}$ is an increasing sequence that converges to $1$.  There is a $\delta-$chian, respect to $F_{2}$, $\frac{1}{2}=y_{0},...y_{N}=1$ from $\frac{1}{2}$ to $1$ and a $\delta-$chain, respect to $F_{1}$, $0=y_{N+1},...,y_{N+K}=\frac{1}{2}$ from $0$ to $\frac{1}{2}$. Hence $x_{0},...,x_{N+K}$ is an $\epsilon-$chain from $x$ to $x$. Up to this point, $x$ is a chain recurrent for  $\mathcal{F}.$
\end{example}
\begin{figure}[htb!]
\label{asd}
\begin{center}
\begin{tabular}{c}
\includegraphics[width=7 cm , height= 6cm]{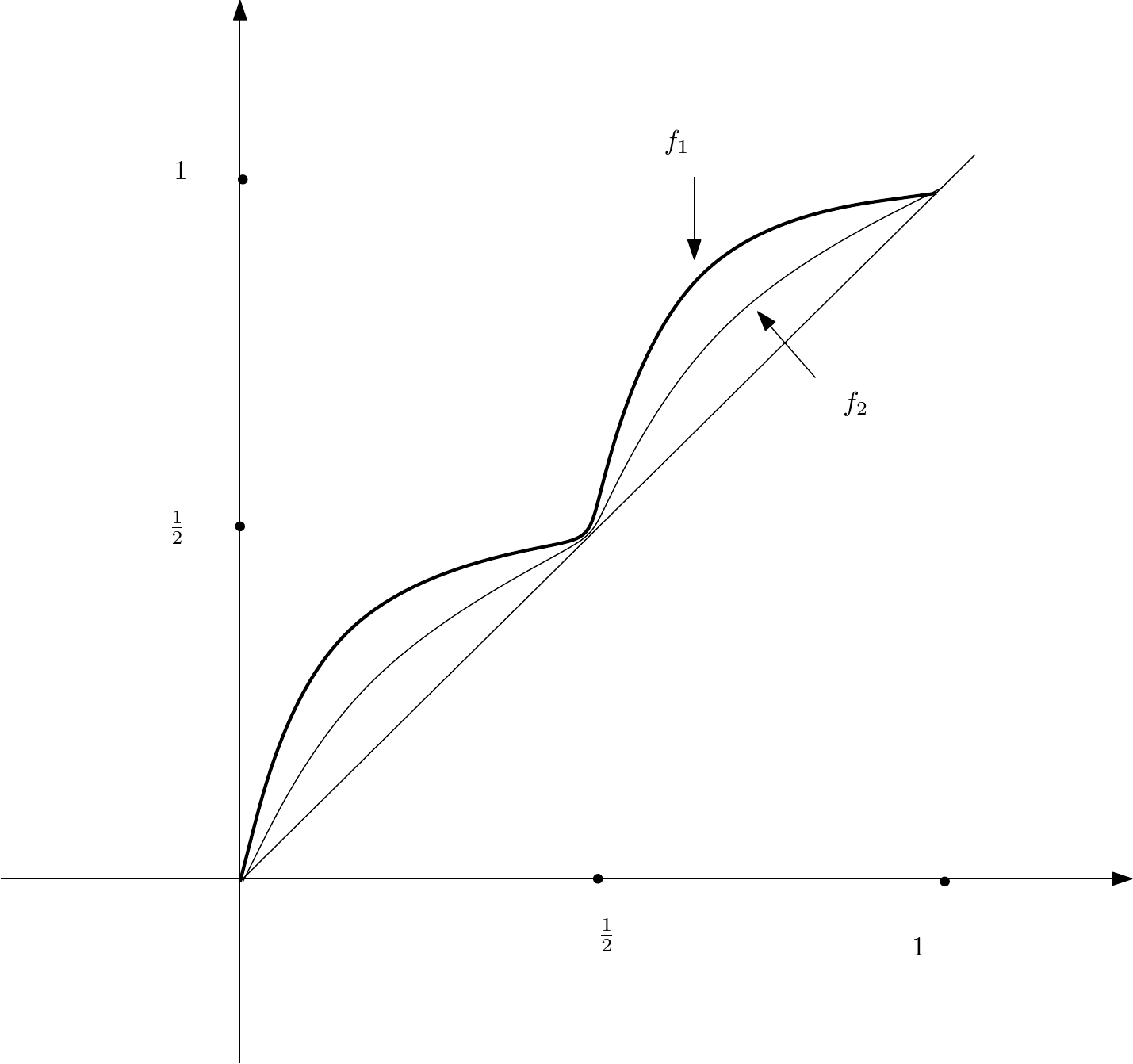}\\
\end{tabular}
\end{center}
\caption{}
\end{figure}
This is known that, in original discrete dynamical systems, that neither asymptotic average nor shadowing property implies the other \cite{[HZ],[RA]}.
The following example shows that in general the asymptotic average shadowing property need not imply the shadowing property for IFS.
\begin{example}\label{e9}
Let $f_{1}, f_{2}: [0,1]\rightarrow [0,1]$ are two continuous map such that $f_{1}(x)>f_{2}(x)>x$ if and only if $x\notin \{0,\frac{1}{2},1\}$ (see Figure $1$ for more details).

Consider the IFS, $\mathcal{F}=\{[0,1]; f_{1},f_{2}\}$. \\
This is clear that for any $\delta>0$ there exist a $\delta$-pseudo orbit from $0$ to $1$, but for every $x\in [0,1]$ and $\sigma=\{\lambda_{0},\lambda_{1},...\}\in\{0,1\}^{\mathbb{Z}_{+}}$ we have $\{\mathcal{F}_{\sigma_{i}}(x)\}_{i\geq0}\subset [0,\frac{1}{2}]$ or $\{\mathcal{F}_{\sigma_{i}}(x)\}_{i\geq0}\subset [\frac{1}{2},1]$. This imply that $\mathcal{F}$ does not have the shadowing property.
\\
Now, we show that $\mathcal{F}$ has the asymptotic average shadowing property. Suppose $\{x_{i}\}_{i\geq 0}$ is an asymptotic average pseudo orbit of $\mathcal{F}$, so there is a sequence $\sigma=\{\lambda_{0},\lambda_{1},...\}\in\{0,1\}^{\mathbb{Z}_{+}}$ such that
\begin{equation}
\label{fa}\lim_{n\rightarrow\infty}\frac{1}{n}\sum_{i=0}^{n-1}|f_{\lambda_{i}}(x_{i})-x_{i+1}|=0.\end{equation} If $\lim_{i\rightarrow\infty}x_{i}=\frac{1}{2}$,  then obviously the fixed point $\frac{1}{2}$ asymptotically shadows $\{x_{i}\}_{i\geq 0}$ in average. \\Otherwise, there exist $\nu\in (0,\frac{1}{10})$ and $N_{0}>0$ such that $x_{i}\in[0,\frac{1}{2}-\nu]\cup[\frac{1}{2}+\nu, 1]$ for all $i>N_{0}$. With no loss of generality, we assume that $x_{i}\in [\frac{1}{2}+\nu, 1]$ for all $i\geq 0$. For $\epsilon\in(0,\nu)$, put\\
$\mu_{\epsilon}=\min\{f_{m}(x)-x:\frac{1}{2}+\epsilon\leq x\leq 1-\frac{\epsilon}{2},m=1 or 2\}.$\\
Then $\mu_{\epsilon}>0$ and $f_{m}(x)\geq x+\mu_{\epsilon}$ for all $x_{i}\in [\frac{1}{2}+\nu, 1-\frac{\epsilon}{2}]$ and $m\in\{0,1\}.$ Thus\\
$f_{\lambda_{1}}o f_{\lambda_{0}}(x)\geq f_{\lambda_{0}}(x)+\mu_{\epsilon}\geq x+2\mu_{\epsilon}$ where $f_{\lambda_{0}}(x)\in [\frac{1}{2}+\nu, 1-\frac{\epsilon}{2}]$.\\ Let $N\geq \frac{(1-3\epsilon)}{2\mu_{\epsilon}}$; we have
\begin{equation}
\label{fb}
\mathcal{F}_{\sigma_{n}}(x)\in(1-\frac{\epsilon}{2},1]~ \text{for all} ~ n\geq N ~\text{and all} ~x\in [\frac{1}{2}+\epsilon,1].
\end{equation}
By (\ref{fa}) and Lemma \ref{l1}, suppose that  $I_{0}\subseteq \mathbb{Z}_{+}$ is a set of zero density such that
$\lim_{j\rightarrow \infty}|f_{\lambda_{i}}(x_{i})-x_{i+1}|=0$ provided $j\notin I_{0}$. Let \\$I_{N}=\{j:\{jN, jN+1,...,jN+N-1\}\cap I_{0}\neq \emptyset\}$ and $I_{N}^{'}=\bigcup_{j\in I_{N}}\{jN, jN+1,...,jN+N-1\}$.\\ Then $I_{N}^{'}$ has density zero and $\lim_{j\rightarrow \infty}|f_{\lambda_{i}}(x_{i})-x_{i+1}|=0$ provided $j\notin I_{N}^{'}$.\\
Since $f_{1}$ and $f_{2}$ are continuous maps, there is a $\delta>0$ such that $|t-s|<\delta$ implies
\begin{equation}
\label{fc}
\mid f_{\lambda_{k}}o...of_{\lambda_{0}}(s)-f_{\lambda_{k}}o...of_{\lambda_{0}}(t)\mid<\frac{\epsilon}{2N},
\end{equation}
 for all $0\leq k\leq N-1$ and $\lambda_{0},..\lambda_{k}\in\{1,2\}.$\\
Take $\kappa\in(0,\min\{\delta,\mu_{\epsilon},\frac{\epsilon}{2}\})$; there exist an integer $N_{1}>N$ such that \\
$|f_{\lambda_{i}}(x_{i})-x_{i+1}|<\kappa~~~$ for all $j>N_{1}$ and $j\notin I_{N}^{'}$.\\
So, by (\ref{fc}) we have
\begin{equation}
\label{fd}
\mid f_{\lambda_{N}}o...of_{\lambda_{0}}(x_{jN})-x_{(j+1)N}\mid<\frac{\epsilon}{2},
\end{equation}
for all $j>N_{1}$ and $j\notin I_{N}^{'}$.\\
Let $I_{N}^{*}=\bigcup_{j\in I_{N}}\{(j-1)N,(j-1)N+1,...,jN,jN+1,...,jN+N-1\}.$ Then $I_{N}^{*}$ has density zero. It follows from (\ref{fb}) and (\ref{fc}) that \\$x_{(j+1)N+k}\in (1-\epsilon,1],~~~~~0\leq k<N-1$ for all $j>N_{1}$ and $j\notin I_{N}^{*}$. Then
\\
$\lim_{n\rightarrow\infty}\frac{1}{n}\sum_{i=0}^{n-1}|x_{i}-1|\leq \lim_{n\rightarrow\infty}\frac{1}{n}\{Card(I_{N}^{*}(n)).1+[n-card(I_{N}^{*}(n))].\epsilon\}=\epsilon.$\\
Where $I_{N}^{*}(n)=I_{N}^{*}\cap\{0,1,...,n-1\}$. Since $\epsilon>0$ is arbitrary, we have $\lim_{n\rightarrow\infty}\frac{1}{n}\sum_{i=0}^{n-1}|x_{i}-1|=0$. This prove that the fixed point $1$ asymptotically shadows $\{x_{i}\}_{i\geq 0}$ in average.
\end{example}


\end{document}